\documentclass[12pt, reqno]{amsart}
\usepackage{fancyhdr}
\usepackage[body={6.3in,9.3in}]{geometry}
\usepackage{verbatim}
\usepackage[ansinew]{inputenc}

\usepackage{textcomp}
\usepackage{latexsym}
\usepackage[usenames]{color}
\usepackage{amsmath,amsfonts,amssymb,amsthm}
\usepackage{graphicx}
\usepackage{longtable}
\usepackage{bbm}
\usepackage{algorithm, algorithmicx, algpseudocode}

\usepackage{setspace}
\usepackage{array}

\def\qed{\hfill{\raggedleft{\hbox{$\Box$}}} \smallskip}

\def\R{\mathbb{R}}

\def\G{\mathcal{G}}
\def\H{\mathcal{H}}

\def\D{\mathcal{D}}

\DeclareMathOperator{\conv}{conv}
\DeclareMathOperator{\spn}{span}

\theoremstyle{plain} \newtheorem{lem}{Lemma}
\theoremstyle{plain} \newtheorem{prop}[lem]{Proposition}
\theoremstyle{plain} \newtheorem{thm}[lem]{Theorem}
\theoremstyle{plain} \newtheorem{cor}[lem]{Corollary}
\theoremstyle{plain} 
\theoremstyle{plain} 
\theoremstyle{definition} \newtheorem{defn}[lem]{Definition}
\theoremstyle{definition}
\theoremstyle{definition} 
\theoremstyle{definition} 
\theoremstyle{definition}\newtheorem{ex}[lem]{Example}

\newlength\savedwidth

\usepackage{tikz}
\usepackage{cite}
\usepackage{enumerate}
\usepackage{lscape}
\usepackage{pdflscape}
\def\1{\mathbf{1}}

\author{Ngoc Mai Tran}
\address{Department of Statistics, UC Berkeley, CA 94720, USA}
\email{tran@stat.berkeley.edu}
\thanks{The author would like to thank Bernd Sturmfels for helpful discussions. This research was supported by the DARPA Deep Learning program (FA8650-10-C-7020)}

\title{Polytropes and tropical eigenspaces: cones of linearity}

\begin{document}

\begin{abstract} The map which takes a square matrix $A$ to its polytrope is piecewise linear. We show that cones of linearity of this map form a polytopal fan partition of $\mathbb{R}^{n \times n}$, whose face lattice is anti-isomorphic to the lattice of complete set of connected relations. This fan refines the non-fan partition of $\R^{n \times n}$ corresponding to cones of linearity of the eigenvector map. Our results answer open questions in a previous work with Sturmfels \cite{ST11} and lead to a new combinatorial classification of polytropes and tropical eigenspaces.
\end{abstract}

\maketitle
\section{Introduction}
Consider doing arithmetics in the \emph{max-plus} tropical algebra $(\mathbb{R}, \oplus, \odot)$ defined by $u\oplus v = \max\{u, v\}, u \odot v = u + v$. In this case, a matrix $A \in \R^{n \times n}$ has a unique tropical eigenvalue $\lambda(A) \in \R$ \cite{bernd.trop, butkovic}. 
Its tropical eigenspace $Eig(A)$ is the set of vectors $x$ in the tropical torus $\mathbb{TP}^{n-1} := \mathbb{R}^n / \mathbb{R} \cdot (1, \ldots, 1)^T$ which satisfy the tropical eigenvector-eigenvalue equation
\vskip-1em
\begin{equation}\label{eqn:evec}
 \max_{1 \leq j \leq n} (A_{ij} + x_j) = \lambda(A) + x_i, \hspace{1em} \mbox{ for each } i = 1, \ldots, n.
 \end{equation}
The set $Eig(A)$ recently found applications in pairwise ranking \cite{elsner10, nmt}, where it appears as a special subset of solutions to the equation
\vskip-1em
\begin{equation}\label{eqn:pol}
 \max_{1 \leq i, j \leq n} (A_{ij} + x_j - x_i) = \lambda(A).
 \end{equation}
The set of $x \in \mathbb{TP}^{n-1}$ satisfying (\ref{eqn:pol}) is convex both in the tropical and ordinary sense. Following Joswig and Kulas \cite{JoswigK10}, we term it the \emph{polytrope of $A$}, denoted $Pol(A)$. 
Both $Eig(A)$ and $Pol(A)$ are tropical polytopes, that is, the tropical convex hull of finitely many points in $\mathbb{TP}^{n-1}$, as in \cite{DS04, SS08}. We shall abuse terminologies and identify $Eig(A)$ and $Pol(A)$ with their tropical extreme points. 

This paper studies conditions under which the piecewise linear maps $A \mapsto Eig(A)$ and $A \mapsto Pol(A)$ are given by linear functionals in the entries of $A$. In previous joint work with Sturmfels \cite{ST11}, we proved the existence of $\mathcal{P}_n$, a partition of $\R^{n \times n}$ into polyhedral cones on which the map $A \mapsto Eig(A)$ is linear. The open cones of $\mathcal{P}_n$ are in bijection with the set of \emph{connected functions}. The main results of this paper, stated below, parallel and extend these theorems. The notations will be fully defined later.
%In this paper, we provide combinatorial interpretations for the lower dimensional cones of $\mathcal{P}_n$, and prove analogous results for the polytrope map $A \mapsto Pol(A)$. 

%--- STATE THEOREM HERE
\begin{comment}
\begin{thm}\label{thm:buddy} There exists a polytopal fan $\mathcal{F}_n$ on $\mathbb{R}^{n \times n}$ that is a refinement of the partition $\mathcal{P}_n$ appearing in \cite{ST11} such that:
\begin{enumerate}[(a)]
	\item In the relative interior of the cones in $\mathcal{F}_n$, the polytrope map $A \mapsto Pol(A)$ is given by a unique set of linear functionals in the entries of $A$. 
	\item There exists a lattice anti-isomorphism $\Psi^{-1}$ from the face lattice of $\mathcal{F}_n$ to the lattice of complete sets of connected relations, denoted $\mathcal{CF}[n]$. The coarsening of $\Psi^{-1}$ to $\mathcal{P}_n$ produces an order-preserving bijection from the closed cones of $\mathcal{P}_n$ to the poset of compatible sets of connected relations, denoted $\mathcal{CP}[n]$.
\end{enumerate}
\end{thm}
\end{comment}
\begin{thm}\label{thm:buddy}
%Let $\partial \mathcal{K}_n$ be the flattening of $\mathcal{K}_n$ along the all-one vector $\1\1^T$. Define
%\begin{equation}\label{eqn:fn} 
%\mathcal{F}_n = \partial \mathcal{K}_n + \spn\{\1\1^T\} + V_n.
%\end{equation}
There exists a fan $\mathcal{F}_n$ such that in the relative interior of its cones, the polytrope map $A \mapsto Pol(A)$ is given by a unique set of linear functionals in the entries of ~$A$. Furthermore, $\mathcal{F}_n$ is the normal fan of an $n(n-1)$-dimensional polytope in $\R^{n \times n}$, whose face lattice is isomorphic to the lattice of complete sets of connected relations, denoted $\mathcal{CF}[n]$. 
\end{thm}
%\subsection*{Contributions} 
Our proof provides an explicit formula for $\mathcal{F}_n$, see (\ref{eqn:fn}). Theorem ~\ref{thm:buddy} addresses all the unresolved issues in \cite{ST11}, namely that of a combinatorial indexing for the lower dimensional cones of $\mathcal{P}_n$, and the existence of a natural fan refinement which preserves information on the eigenspace. 

\begin{cor}\label{cor:buddy}
The fan $\mathcal{F}_n$ in Theorem \ref{thm:buddy} refines the partition $\mathcal{P}_n$ appearing in \cite{ST11}. The coarsening to $\mathcal{P}_n$ of the lattice isomorphism in Theorem \ref{thm:buddy} produces an order-preserving bijection from the closed cones of $\mathcal{P}_n$ to the poset of compatible sets of connected relations, denoted $\mathcal{CP}[n]$.
\end{cor}

Theorem \ref{thm:buddy} converts the study of cones in $\mathcal{F}_n$ and $\mathcal{P}_n$ to the study of $\mathcal{CF}[n]$ and $\mathcal{CP}[n]$, which essentially are collections of graphs with constraints. From these sets of graphs, one can recover the defining equations and inequalities of the corresponding cones by applying the lattice isomorphism, which is a simple computation given in Algorithm \ref{alg:psi}. Cone intersections, for example, correspond to taking the join of two elements in $\mathcal{CF}[n]$ as described in Algorithm ~\ref{alg:join.on.CF}. Coarser information on a cone such as its codimension is given by an explicit formula (Proposition \ref{prop:codim}). 

\subsection*{Connections to literature}
As noted in \cite{JoswigK10}, it follows from \cite[Proposition 18]{DS04} that all polytropes arise as $Pol(A)$ for some matrix $A$. This was also known to Sergeev, Schneider and Butkovic \cite{SSB09}, and further studied in \cite{SS08}. Full-dimensional polytropes, also known as tropical simplices, are of great importance to tropical polyhedral geometry, see \cite{JoswigK10} and references therein. 

The face lattice of polytropes provides a natural combinatorial classification which coincides with Green's D-relation on the semigroup of square tropical matrices \cite{kambite}. On the other hand, one can declare two polytropes to be equivalent if the corresponding matrices have finite entries and they lie in the relative interior of the same cone in $\mathcal{F}_n$. We term this the \emph{graphical type} of a polytrope, and the former its \emph{combinatorial type}. They are not directly comparable. We do not know if there exists an indexing set equivalent to $\mathcal{CF}[n]$ for the combinatorial type of polytropes. Such a set may shed light on the semigroup structure of tropical matrices, a topic intimately tied in with the `tropical' (ie: Brazilian) origins of tropical geometry in the work of Imre Simon \cite{simon}. 
%Finally, we would like to mention the connection between our paper and that of Allamigeon, Gaubert and Goubault \cite{AGG09}, where they provided a combinatorial characterization of extreme tropical vertices of tropical polyhedra, which are vector solutions to the inequality $A\odot x \leq B \odot x$ for a given pair of matrices $A, B$. For the special case of polytropes, their result states that each extreme tropical vertex of $Pol(A)$ corresponds to a strongly connected component which acts as the sink in a directed multigraph. This is in parts the content of Theorem \ref{thm:buddy}. Our main contribution is a precise characterization of how the various multigraphs on these disjoint sinks fit together. It would be interesting to extend our approach towards a classification of tropical polyhedra. Concretely, given a pair of matrices $A, B \in \mathbb{R}^{n \times n}$, what are the cones of linearity of the map which takes $A, B$ to the extreme tropical vertices of the tropical polyhedron $\{x: A \odot x \leq B \odot x\}$?
%In the light of the connections between tropical eigenspaces, polytropes, semigroups of tropical matrices and tropical polyhedra, we hope that this combinatorial approach can be used to further the study of these areas. 

\subsection*{Organization} In Section \ref{sec:basic} we explicitly construct $\mathcal{F}_n$ and prove that it is the desired polytopal fan. In Section \ref{sec:prelim} we define $\mathcal{CF}[n]$ and $\mathcal{CP}[n]$, the complete and compatible sets of connected relations, respectively. Section \ref{sec:proof} proves the remaining statement of Theorem ~\ref{thm:buddy} and Corollary \ref{cor:buddy} by constructing the lattice anti-isomorphism $\Psi$ from $\mathcal{CF}[n]$ to the face lattice of $\mathcal{F}_n$. We give an algorithm to compute $\Psi$ in Section ~\ref{sec:codim}, prove a formula which gives the co-dimension of the cone from its graphical encoding, and discuss symmetries of the fan $\mathcal{F}_n$. We conclude with examples in Section \ref{sec:ex}. 

\subsection*{Notation} Throughout this paper, a \emph{graph} is a directed graph, allowing self-loops but not multiple edges. By a \emph{connected} graph we mean weakly connected. The \emph{subgraph rooted at a node $u$} is the set of nodes and edges belonging to paths which flow into $u$. A collection of nodes $S$ in a graph $G$ is a \emph{strong component} if the induced subgraph on $S$ is strongly connected. A strong component $S$ is a \emph{sink component} if there are no directed edges from $S$ to its complement. We use the term \emph{multigraph} to mean a directed graph with multiple distinct arcs, that is, it consists of an ordered 4-tuple $(V, E, s, t)$ where $V$ is the set of nodes, $E$ is the set of edges, $s: E \to V$ is a map assigning to each edge its source node, and $t: E \to V$ is a map assigning to each edge its target node. The \emph{contraction} $\tilde{G}$ of a connected graph $G$ is a multigraph whose nodes are indexed by the sink components of $G$, and whose edge set, source and target maps are induced by the edges of $G$. An \emph{in-directed tree} is a tree whose edges are oriented towards the root. The \emph{graph of a matrix $A \in \R^{n \times n}$} is the weighted graph with edge weights $A_{ij}$. For a path $P_{ii'}$ from $i$ to $i'$, let $|P_{ii'}|$ denote the number of edges, $A(P_{ii'})$ denote the sum of all edge weights along the path in the graph of $A$. We say that two paths are disjoint if they do not share any edges. For disjoint paths $P_{ii'}, P_{i'j}$, we write $P_{ii'} + P_{i'j}$ for their concatenation. For two fans $\mathcal{F}_1, \mathcal{F}_2$, let $\mathcal{F}_1 \wedge \mathcal{F}_2$ denote their common refinement. %For a polyhedron $P$, write $\mathcal{N}(P)$ for its normal fan. 

%----------- STARTING OF A NEW SECTION
\section{Construction of $\mathcal{F}_n$}\label{sec:basic}

In this section we give an explicit construction of $\mathcal{F}_n$, and show that it is polytopal. We first need a small result in polyhedral geometry which gives one way of constructing a polytopal fan.  

\begin{defn} Let $\mathcal{F}$ be a fan refinement of a convex, full-dimensional, pointed polyhedral cone ~$C$ in $\R^n$, $v \in \R^n$ a vector in its interior. A cone of $\mathcal{F}$ is an \emph{outer cone} if it is a subset of the boundary of $C$. The \emph{flattening of $\mathcal{F}$ along $v$} is the fan in $\R^{n-1}$ whose cones are the images of the outer cones of $\mathcal{F}$ projected onto the orthogonal complement of $v$.
\end{defn}

\begin{lem}\label{lem:polytope} Suppose $\mathcal{F}$ is the normal fan of some convex polyhedron $P$ in $\R^n$ with a pointed, full-dimensional recession cone. Let $v$ be a vector in the interior of $\mathcal{F}$. Then the flattening of $\mathcal{F}$ along $v$ is a polytopal fan in $\R^{n-1}$. 
\end{lem}
\begin{proof} Implicitly the lemma claimed that $\mathcal{F}$ is a fan refinement of a full-dimensional, pointed polyhedral cone. To see this, let $P$ be the polyhedron in the lemma. Write $P = Q + C$ where $C$ is its recession cone, $Q$ a polytope. By definition, the polar $C^\circ$ is the normal fan of $C$. Since $C$ is full-dimensional, the facet normals of $Q$ lie in the interior of $C^\circ$. Thus the normal fan of $P$ is a fan refinement of the polyhedral cone $C^\circ$. 

An $(n-1)$-dimensional unbounded face of $P$ is a face of $Q$ plus a face of $C$, hence its normal vector is an outer ray of $\mathcal{F}$. An $(n-1)$-dimensional bounded face of $P$ is necessarily a face of $Q$, whose facet normal must lie in the interior of $C^\circ$. Thus the outer rays of $\mathcal{F}$ are precisely the facet normals of the unbounded faces of $P$. 

We now explicitly construct the polytope in question. Let $\mathcal{H}$ denote the orthogonal complement of $v$. Under the convention that the facet normals point outwards, the linear functional $x \mapsto \langle v, x \rangle$ on $\R^n$ is bounded above on $P$, and bounded below on $Q$. Thus there exists $M$ sufficiently small such that $\langle v, x \rangle > M$ for all $x \in Q$. Let $\mathcal{H}_M = \mathcal{H}+M$. View $P_M = P \cap \mathcal{H}_M$ as a polytope in the $(n-1)$-dimensional affine subspace $\mathcal{H}_M$. We claim that $P_M$ is the polytope needed. Indeed, by construction, the $i$-dimensional faces of $P_M$ are the unbounded $(i+1)$-dimensional faces of $P$ intersected with $\mathcal{H}_M$. Thus it is sufficient to show that the projected outer rays of $\mathcal{F}$ along $v$ are the facet normals of the corresponding faces in $P_M$. 
Let $x$ be a vector in an $(n-2)$-dimensional face of $P_M$, $f$ the corresponding outer ray of $\mathcal{F}$, and $\Pi f$ the orthogonal projection of $f$ onto $\mathcal{H}$. Since $f - \Pi f$ is orthogonal to $\mathcal{H}$, $\langle x, f - \Pi f\rangle = 0$. Thus $\langle x, \Pi f\rangle = \langle x, f \rangle = 0$. Therefore, the rays $\Pi f$ are the facet normals of $P \cap \mathcal{H}_M$. 
\end{proof}

Let $\Gamma_n$ be the set of simple directed cycles on the complete digraph on $n$ nodes, including self-loops. For $C \in \Gamma_n$, let $\chi_C \in \{0,1\}^{n^2}$ be its incidence vector. Let $\mathcal{J}_n$ be the polyhedral cone of matrices in $\R^{n \times n}$ with no positive cycles
$$ \mathcal{J}_n = \{A \in \R^{n \times n}: A\cdot \chi_C \leq 0, C \in \Gamma_n \}. $$
Its lineality space $V_n$ has dimension $(n-1)$, consisting of matrices with zero-cycle sum.
$$ V_n = \{A \in \R^{n \times n}: A\cdot \chi_C = 0, C \in \Gamma_n\}. $$
For any $A \in \mathcal{J}_n$, define $ A^+ := A \oplus A^2 \oplus \ldots, \mbox{ and }\hspace{0.5em} A^\ast := I \oplus A^+, $
where $I$ is the matrix with $0$ on the diagonal and $-\infty$ elsewhere. 

The map $A \mapsto A^\ast$ is the \emph{Kleene star} map, and $A^\ast$ is the \emph{Kleene star} of $A$ \cite{SS08}. This map is elementwise finite if an only if $A \in \mathcal{J}_n$. Kleene stars provide a very neat characterization of extreme points of $Eig(A)$ and $Pol(A)$, as well as lending them the combinatorial interpretations crucial to our construction of $\mathcal{F}_n$. 
For a proof of the following result, see \cite[\S 3]{SS08}.
\begin{prop}\label{prop:kleene}
For an $n \times n$ matrix $A$, define $\bar{A} := A \odot (-\lambda(A))$. Then the extreme tropical vertices of $Pol(A)$, up to tropical scaling, are precisely the columns of $\bar{A}^\ast$, and the extreme tropical vertices of $Eig(A)$ are precisely the common columns of $(\bar{A})^+$ and $\bar{A}^\ast$. %The indices of such columns are called the \emph{critical nodes} in the graph of $A$. 
%Then $Eig(A) = Eig(\bar{A}^+)$, $Pol(A) = Pol(\bar{A}^+) = Pol(\bar{A}^\ast)$
\end{prop}

Thus to study the polytrope map, we need to understand cones of linearity of the eigenvalue map $A \mapsto \lambda(A)$ and the Kleene star map. The former is a classic result of Cunninghame-Green \cite{Cg62}, which we reformulate in view of Lemma \ref{lem:polytope} as follows.

\begin{prop}\label{prop:Nn} Let $\partial \mathcal{J}_n$ be the flattening of $\mathcal{J}_n$ along $-\1\1^T$. Define
$$ \mathcal{N}_n = \partial \mathcal{J}_n + \spn\{\1\1^T\} + V_n, $$
where $+$ denotes direct sum of subspaces. Then $\mathcal{N}_n$ is a polytopal fan (the normal fan of a polytope) in $\R^{n \times n}$. Its cones are precisely the cones of linearity of the eigenvalue map. 
\end{prop}
\begin{proof}
The rays of the polar $\mathcal{J}_n^\circ$ are the incidence vectors $\{\chi_C: C \in \Gamma_n\}$. Since all faces of $\mathcal{J}_n$ are unbounded except the $0$ vertex, we can choose any constant $M < 0$ for the flattening, say, $M = -1$. Now $\mathcal{J}_n^\circ \cap \mathcal{H}_{-1}$ is the normalized cycle polytope
$$ \mathcal{C}_n = \conv\left\{\frac{1}{|C|}\chi_C: C \in \Gamma_n\right\}.$$
By Lemma \ref{lem:polytope}, $\partial \mathcal{J}_n$ is the normal fan of $\mathcal{C}_n$ viewed as an $(n^2 - n)$-dimensional polytope in the affine subspace $\mathcal{H}_{-1}$. One can check that $\R^{n\times n} = \mathcal{H}_{-1} + \spn\{\1\1^T\} + V_n$, thus $\mathcal{N}_n$ is the normal fan of $\mathcal{C}_n$ in $\R^{n \times n}$. By \cite{Cg62}, the eigenvalue function $A \mapsto \lambda(A)$ is the support function of $\mathcal{C}_n$ in $\R^{n \times n}$. Thus $\mathcal{N}_n$ is precisely the fan of linearity of this map.  
\end{proof}
We now consider the Kleene star map $A \mapsto A^\ast$. By definition, $A^\ast_{ij}$ is the value of the maximal path from $i$ to $j$ in the graph of $A$. For a fixed column $i^\ast$, over $\mathcal{J}_n$, the map $A \mapsto A^\ast_{\cdot i^\ast}$ is precisely the solution to the classical linear program of single-target longest path (more commonly formulated as single-source shortest path) \cite{networkBook}.
\vskip -1em
\begin{align*}
\mbox{Maximize } & \sum_{i,j=1}^n A_{ij}x_{ij}  \\
\mbox{ subject to } & x_{ij} \geq 0, \sum_{j=1}^n x_{ji} - \sum_{j=1}^n x_{ij} = \left\{ 
\begin{array}{clc}
n - 1 & \mbox{ for } & i = i^\ast \\
-1 & \mbox{ for } & i \neq i^\ast 
\end{array}
\right..
\end{align*}

This linear program is bounded if and only if $A \in \mathcal{J}_n$. Its constraint set is a polyhedron in $\mathbb{R}^{n \times n}$ with lineality space $V_n$, whose vertices are in bijection with in-directed spanning trees on $n$ nodes with root $i^\ast$ \cite{networkBook}. %, and the value of an edge is the number of paths in the tree which uses that edge. 
Thus the cones of linearity of $A \mapsto A^\ast_{\cdot i^\ast}$ form a fan partition of $\mathcal{J}_n$. Denote this fan $\mathcal{K}_{n,i^\ast}$. Take the common refinement over all columns $i^\ast$, we obtain the fan
$$\mathcal{K}_n = \bigwedge_{i^\ast=1}^n\mathcal{K}_{n,i^\ast}.$$
\begin{cor}\label{cor:kn}
$\mathcal{K}_n$ is the fan of linearity of the Kleene star map on $\mathcal{J}_n$. Furthermore, it is the normal fan of a polyhedron in $\R^{n \times n}$.
\end{cor}
\begin{proof}
The first statement follows from Proposition \ref{prop:kleene}. For each $i^\ast$, $\mathcal{K}_{n,i^\ast}$ is the normal fan of a polyhedron. Thus $\mathcal{K}_n$ is the normal fan of the Minkowski sum of the corresponding polyhedra, see \cite[\S 2.3]{bernd.trop}. 
\end{proof}

\emph{Proof of Theorem \ref{thm:buddy}, part (a).}
Let $\partial \mathcal{K}_n$ be the flattening of $\mathcal{K}_n$ along $-\1\1^T$. Define
\begin{equation}\label{eqn:fn} 
\mathcal{F}_n = \partial \mathcal{K}_n + \spn\{\1\1^T\} + V_n.
\end{equation}
By Proposition \ref{prop:Nn} and Corollary \ref{cor:kn}, $\mathcal{F}_n$ is the fan of linearity of the polytrope map $A \mapsto Pol(A)$. By Lemma \ref{lem:polytope}, it is a polytopal fan in $\R^{n \times n}$. \qed

The fan $\mathcal{F}_n$ refines $\mathcal{N}_n$, and inherits its lineality space $\spn\{\1\1^T\} + V_n$. We often identify $\mathcal{F}_n$ with $\partial \mathcal{K}_n$, $\mathcal{N}_n$ with $\partial \mathcal{J}_n$, as in Example ~\ref{ex:n=2} and Example \ref{ex:n3}. 

\begin{ex}[$n=2$]\label{ex:n=2} Here $\mathcal{K}_2$ is the trivial fan, and $\mathcal{N}_2 = \mathcal{P}_2 = \mathcal{F}_2$. The fan $\mathcal{N}_2$ has three cones of codimension 0, three cones of codimension 1, and one cone of codimension ~2. Express the matrix $A$ as the vector $(x,y,z,w)$, where $x = A_{11}, y = A_{22}, z = A_{12} + A_{21}, w = A_{12} - A_{21}$. Then $V_2 = \{(x,y,z,w): x=y=z=0\}$, $\mathcal{J}_2 = \{x, y, z \leq 0, w = 0\}$ is the negative orthant, and $\partial \mathcal{J}_2$ is the image of the three quadrants $\{x = w = 0, y,z \leq 0\}$, $\{y = w = 0, x,z \leq 0\}$ and $\{z = w = 0, x,y \leq 0\}$ projected onto the subspace defined by $\{w = 0, x + y + z = 0\}$. This projection results in a fan partition of $\R^2$ seen in Figure \ref{fig:n2}. 

\begin{figure}[h]
	\begin{center}
	 \includegraphics[width=0.8\textwidth]{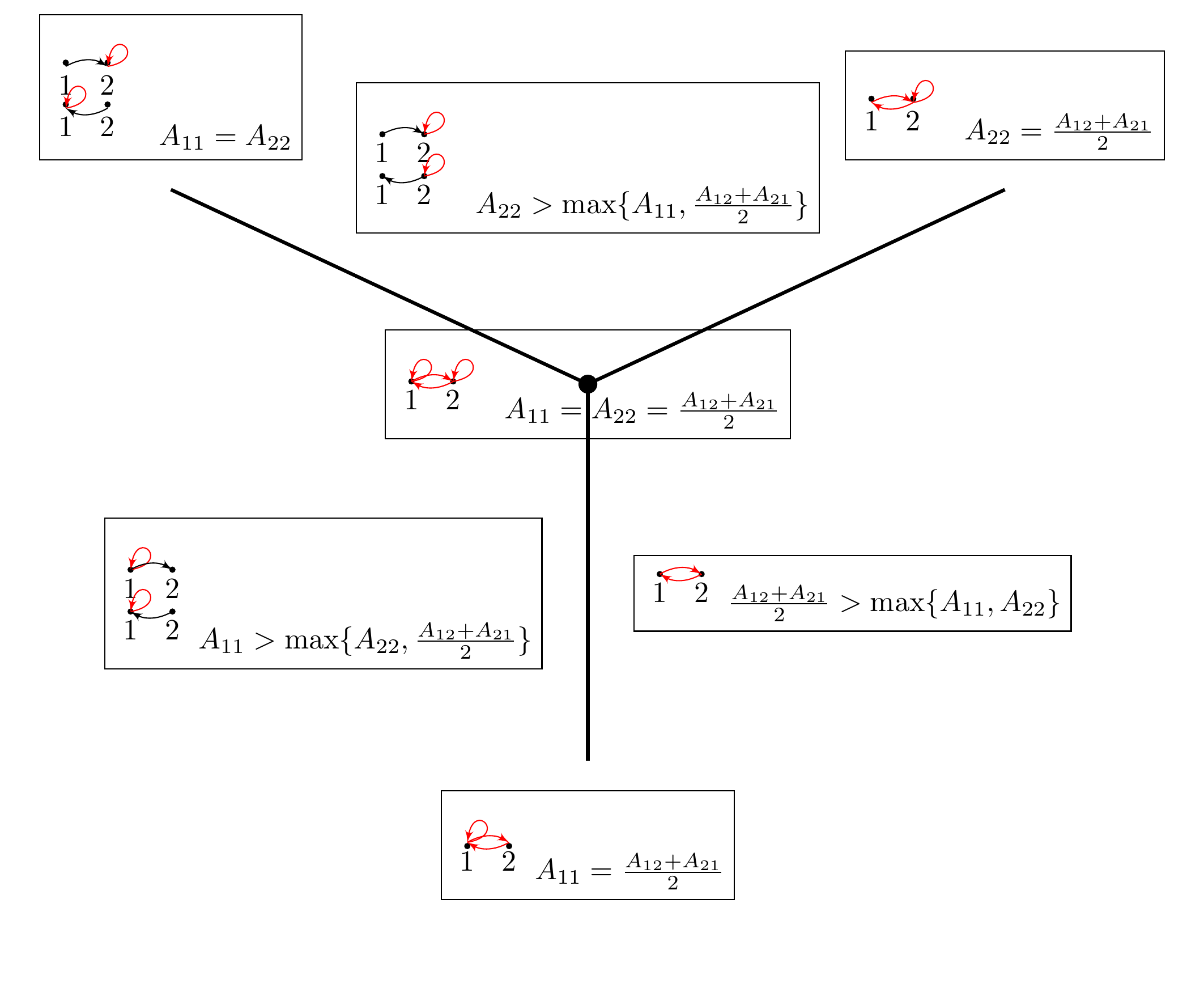}
	 \vskip -0.4cm
	\caption{Cones in the partition $\mathcal{N}_2 = \mathcal{P}_2 = \mathcal{F}_2$ of $\R^{2 \times 2}$, their defining equations and inequalities, and their associated maximal paths. Figure is drawn modulo lineality space, and thus represents the fan $\partial \mathcal{J}_2 = \partial \mathcal{K}_2$ of $\R^2$. Edges in the critical cycles (cycles with maximal mean) are shown in red.}
	\label{fig:n2}
	\end{center}
\end{figure}
\end{ex}
We shall revisit this example in Section \ref{sec:ex} when discussing complete sets of connected relations. For the rest of this paper we study how the sets of graphs indexing cones in Figure \ref{fig:n2} can be generalized to indexing sets for cones of $\mathcal{P}_n$ and $\mathcal{F}_n$ for $n \geq 3$. 

\section{complete and compatible sets of connected relations}\label{sec:prelim}

For fixed $i^\ast$, the open cones of $\mathcal{K}_{n, i^\ast}$ are in bijection with in-directed spanning trees on $n$ nodes with root $i^\ast$, whose edges form the optimal paths from $j$ to $i^\ast$ for all nodes $j \in \{1,\ldots,n\}, j \neq i^\ast$ \cite{networkBook}. One can show that the open cones of $\mathcal{K}_n$ are in bijection with collections of $n$ in-directed spanning trees on $n$ nodes, one rooted at each node, which satisfy following \emph{path compatibility} condition.
\begin{lem}[\cite{networkBook}]\label{lem:path.comp}
If the path $i_1 \to i_2 \to \ldots \to i_k$ is a longest path from node $i_1$ to node $i_k$, then for every $q = 1, \ldots, {k-1}$, the subpath $i_q \to i_{q+1} \to \ldots \to i_k$ is a longest path from node $i_q$ to node $i_k$. 
\end{lem}
Cones of $\mathcal{J}_n$ with codimension one are in bijection with cycles on the complete graph. Their intersections with the open cones of $\mathcal{K}_n$ form the open cones of $\mathcal{F}_n$. Thus, the later are in bijection with \emph{complete connected functions} (see Definition \ref{defn:ccr}). Informally, a complete connected function consists of $n$ trees unioned with a cycle, one rooted at each node, which satisfy the path compatibility condition. 

Many properties of the open cones of $\mathcal{N}_n$ also hold for those of $\mathcal{F}_n$. For example, it can be shown that neighboring full-dimensional cones of $\mathcal{F}_n$ differ by one edge. However, due to the path compatibility condition, it is less obvious what pairs of edges can be swapped to move to a neighboring cone. Similarly, one can compute cone intersection by taking the graph union between pairs of trees with the same root node. However, in doing so, new paths (such as a cycle) may appear, which by the path compatibility condition should also be present in other trees.

Therefore, one needs an efficient indexing set for all closed cones of $\mathcal{F}_n$, and a join operation that corresponds to cone intersections. With the hindsight of Theorem \ref{thm:buddy}, we chose to start off by defining the desired indexing set, the complete sets of connected relations. We first introduce connected relations, the building block for our indexing set. In relation to \cite{ST11}, they are generalizations of connected functions. 

\begin{defn}
A \emph{circled tree} on $[n]$ is the graph union of a cycle and an in-directed spanning tree on $n$ nodes, such that it has exactly one cycle.
A \emph{connected relation} on $[n]$ is a non-empty graph which can be written as the graph union of circled trees on $[n]$.
\end{defn}
Let $\mathcal{R}[n]$ denote the set of all connected relations on $[n]$ adjoined with the empty set. It is a lattice ordered by subgraph inclusion, generated by its join-irreducibles, which are circled trees. Connected functions, for example, are strongly connected circled trees. Note that every connected relation has a unique strongly connected sink component. For $G \in \mathcal{R}[n]$, denote its unique sink component by $s(G)$ and the induced subgraph $S(G)$.

%Circled tree naturally arise... 
\begin{defn}\label{defn:ccr} Let $\G = \{G_1, \ldots, G_k\}$ be a set of connected relations, $G_i \in \mathcal{R}[n]$. For node $u$, let $\G(u)$ denote the part $G_i \in \G$ containing $u$ in its sink. The list $\G$ is  a \emph{complete set of connected relations} if:
\begin{enumerate}[(a)]
	\item 	The sinks $s(G_j), j = 1, \ldots, k$ form a partition of $[n]$. %and the same self-loops. (Same self-loops = implied). 
	%In other words, the sink components $s(G_i)$ are mutually disjoint, form a partition of $[n]$, and they are precisely the strongly connected components of each $G_j$. Furthermore, for each fixed $j$, the subgraphs of $G_i$ restricted to $s(G_j)$ are identical for $i = 1, \ldots, k$.  
%	\item If a cycle $C \subseteq G_i$ for some $i \in [k]$, then $C \subseteq G_j$ for all $j \in [k]$, and $C \subseteq S(G_{i'})$ for some $i' \in [k]$. 
	\item For each pair $i, j \in \{1, \ldots, k\}$, let $S_{ij}$ be the subgraph of the contraction $\widetilde{G}_j$ of $G_j$ rooted at the node indexed by $s(G_i)$. Then $S_{ij}$ equals the induced subgraph in $\widetilde{G}_i$. 
	\item For every triple of distinct nodes $u, v, w \in \{1, \ldots, n\}$, if there exists paths $P_{wv}, Q_{wu}$ in $\G(v)$ such that $Q_{wu}$ is not a subpath of $P_{wv}$, then $P_{wv}$ cannot appear in $\G(u)$. 
\end{enumerate}
Let $\mathcal{CF}[n]$ denote the collection of completed sets of connected relations. 
\end{defn}
Condition (b) is the analogue of the path compatibility condition in Lemma \ref{lem:path.comp}. Note that it implies that the contractions $\widetilde{G}_i$ have the same node set, and they are indexed by the set of sinks $s(G_j), j = 1, \ldots k$. The need for condition (c) is discussed after the proof of Lemma \ref{lem:psi.not.im}. 

\begin{defn}
Write $\G = (\D, \mathcal{E})$ when we decompose $\G$ into the set $\mathcal{D}$ of connected relations whose sinks each contain at least one edge, and the set $\mathcal{E}$ of connected relations whose sinks each contain exactly one node and no edges. We say that $\D$ is a \emph{compatible set of connected relations} and let $\mathcal{CP}[n]$ denote all such $\D$. If $\D$ consists of exactly one connected function, we say that $\G$ is a \emph{complete connected function}.  
\end{defn}
%Note that $\mathcal{E}$ can be an empty set but not $\D$, since at least one strongly connected component of each $G_i$ necessarily contains a cycle. 

\begin{ex}\label{ex:G}
Figure \ref{fig:cp} shows an incomplete set of connected relations $\G = (D,E)$ on 3 nodes. Here $s(D) = \{2\}$, $s(E) = \{3\}$. Since $\G$ is missing a part with the strong component $\{1\}$ as its sink, it is not complete.
\begin{figure}[h]
	\begin{center}
	 \includegraphics{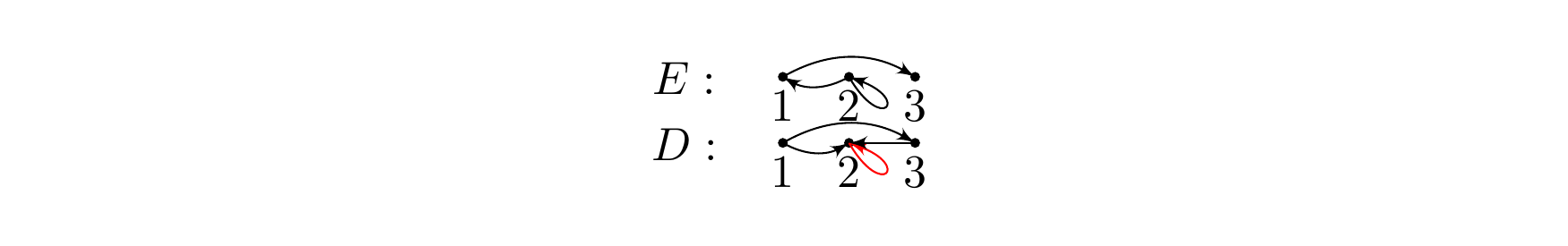}
	 \vskip -0.6cm
\caption{An incomplete set of connected functions $\G = (D,E)$.}
\label{fig:cp} 
	\end{center}
\end{figure}
\end{ex}
\begin{ex}
We can complete $\G$ in Example \ref{ex:G} by adding in a connected relation with sink $\{1\}$ that does not violate condition (b) of Definition \ref{defn:ccr}. Figure \ref{fig:cp.allowed} shows all such compatible connected relations, and Figure \ref{fig:cp.not.allowed} shows some incompatible connected relations. 
In the first case, the subgraph rooted at $3$ consists of the edge $2 \to 3$, which is not a subgraph of $E$. The second violates condition (a) since the sinks do not form a partition of $[n]$. The third violates condition (b) as $D$ and $E$ do not contain the self-loop at 3. 

\begin{figure}[h]
	\begin{center}
	 \includegraphics{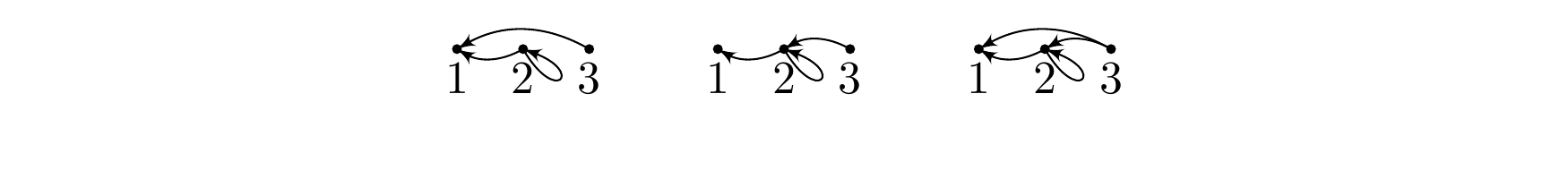}
	 \vskip -1cm
\caption{All three possible connected relations that can complete the set $\G$ of Figure \ref{fig:cp}.}
\label{fig:cp.allowed} 
	\end{center}
\end{figure}

\begin{figure}[h]
	\begin{center}
	 \includegraphics{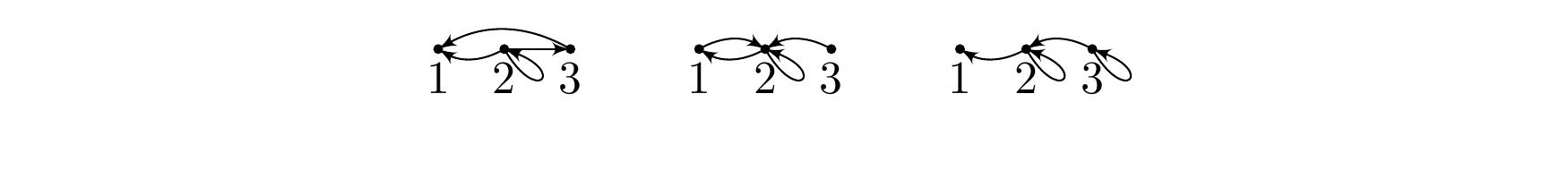}
	 \vskip -1cm
\caption{Some connected relations incompatible to $\G$ of Figure \ref{fig:cp}.}
\label{fig:cp.not.allowed} 
	\end{center}
\end{figure}
\end{ex}

%------------------------
\section{The lattice isomorphism}\label{sec:proof}
In this section we prove the rest of Theorem \ref{thm:buddy}, namely that the face lattice of $\mathcal{F}_n$ is anti-isomorphic to the lattice $\mathcal{CF}[n]$. The proof consists of three steps. First, we show that $\mathcal{CF}[n]$ is a join semilattice generated by its join irreducibles, which are the complete connected functions. Second, we produce a semilattice homomorphism $\Psi$ from $\mathcal{CF}[n]$ to a collection of cones in $\mathbb{R}^{n \times n}$ ordered by subset inclusion. Third, we show that the open cones of the fan $\mathcal{F}_n$ defined in (\ref{eqn:fn}) are precisely the images of the complete connected functions under $\Psi$. Since $\mathcal{F}_n$ is a fan, its face lattice ordered by subset inclusion is generated by meet-irreducibles, which are the open cones of $\mathcal{F}_n$. Together with the first step, this implies that $\Psi$ is a lattice anti-isomorphism between $\mathcal{CF}[n]$ and cones of $\mathcal{F}_n$. 

\subsection{The semilattice of $\mathcal{CF}[n]$.}
We now show that $\mathcal{CF}[n]$ is a join semilattice generated by its join irreducibles, which are the complete connected functions. For $\G \in \mathcal{CF}[n]$, let $\lambda(\G)$ denote the partition of $[n]$ induced by the sinks of its parts. For a pair $\G, \H \in \mathcal{CF}[n]$, write $\H \geq \G$ if for any pair of nodes $u,v$, the set of paths from $u$ to $v$ in $\G(v)$ is a subgraph of $\H(v)$. It is clear that this defines a partial ordering on $\mathcal{CF}[n]$. 

We now define the join of two elements $\G, \H \in \mathcal{CF}[n]$ to be the output of Algorithm \ref{alg:join.on.CF}. This is an elaborated version of taking graph union and updating the strongly connected sink components at the same time. Broadly speaking, the algorithm iterates over three steps. First, we take the minimal amount of graph union required to satisfy condition (a) of Definition ~\ref{defn:ccr}. In doing so, new cycles may form which may connect two sink components, requiring them to be merged. Thus in step two, we check for all such cycles. In step three, we find all triples violating condition (c) of Definition \ref{defn:ccr}, and resolve them by taking the graph union of the corresponding connected relations. If `actions' happened in steps two and three, we go back to step one. Otherwise, the algorithm terminates and outputs its current set of connected relations. 

Note that after each iteration of line \ref{ln:iterate} the partition $\lambda$ is unique and becomes coarser. Thus the algorithm is well-defined and terminates. 
By construction, the output satisfies Definition \ref{defn:ccr}, and hence it is a complete set of connected relations.

\begin{algorithm}
\caption{The join operation on $\mathcal{CF}[n]$.}\label{alg:join.on.CF}
\begin{algorithmic}[1]
\State Input: a pair $\G, \H \in \mathcal{CF}[n]$. 
\State Output: a set of connected relations called $\G \vee \H$.
\Procedure{Join}{$\G,\H$}
\State Set $B = \emptyset$. 
\State Compute $\lambda$, the finest partition of $[n]$ that $\lambda(\G), \lambda(\H)$ both refine. \label{ln:step.a}
\State \textbf{Step 1:} \For{each component $\lambda_i$ of $\lambda$} \label{ln:compute}
	\State  Compute 	\begin{equation}\label{eqn:join}
		I_i = \bigcup_{u \in \lambda_i}\G(u) \cup\H(u) \cup B
		\end{equation}
\EndFor \label{ln:endfor}
\State \textbf{Step 2:} Set $B'$ to be the union over all $i$ of all cycles $C \subset I_i$ where $C \not\subset S(I_j)$ for any $j$.
\State \textbf{Step 3:} 
\For{each ordered triple of distinct nodes $u, v, w \in [n]$}
	\For{each pairs of paths $P_{wv}$, $Q_{wu}$ in $\G(v)$ that violates Definition \ref{defn:ccr}(c)}
			\State Add to $B'$ the cycles formed by the union of all paths $P_{uv}$ in $\G(v)$ and all paths $Q_{vu}$ in $\G(u)$.
	\EndFor
\EndFor	
\If{$B' = \emptyset$}  \label{ln:b.prime}
	\State \textbf{return} $\mathcal{I} = (I_1, \ldots, I_{|\lambda|})$ \Comment{The output is the list $\mathcal{I}$}
\Else
	\State Set $B = B \cup B'$ 
	\State Compute $\lambda$, the finest partition of $[n]$ such that each component of $\lambda(I_j)$, $j = 1, \ldots, |\lambda|$ and each cycle $C$ in $B'$ is contained in exactly one component of $\lambda$. \label{ln:iterate}
	\State Go to line \ref{ln:compute}. 
\EndIf
\EndProcedure
\end{algorithmic}
\end{algorithm}

\begin{lem}\label{lem:join} $\G \vee \H$ is the join of $\G$ and $\H$ in $\mathcal{CF}[n]$.
\end{lem}
\begin{proof} From construction it is clear that $\G \vee \H \in \mathcal{CF}[n]$ and $\mathcal{I} \geq \G, \H$. To see that it is the smallest such element in $\mathcal{CF}[n]$, let $\mathcal{I}^{(1)}, \mathcal{I}^{(2)}, \ldots \mathcal{I}^{(k)} = \mathcal{I}$ be the sequence of intermediate candidate sets produced after the end of line \ref{ln:endfor} for each time this line is reached. 

Let $\mathcal{I}' \in \mathcal{CF}[n]$ be such that $\mathcal{I}' \geq \G, \H$. Then each component of $\lambda(\G), \lambda(\H)$ is necessarily contained in exactly one component of $\lambda(\mathcal{I}')$. Thus any path in $\mathcal{I}^{(1)}(v)$ is a subgraph of $\mathcal{I}'(v)$, and in particular, any cycle $C$ in $B'$ must be a subgraph of some part $\mathcal{I}'_i$. Since $C$ is a cycle, by Definition \ref{defn:ccr}, it is contained in exactly one component of $\lambda(\mathcal{J})$. Thus any path in $\mathcal{I}^{(2)}(v)$ must be a subgraph of $\mathcal{I}'(v)$. Iterating this argument up the sequence $\mathcal{I}^{(j)}$, $j = 1, \ldots, k$, we see that any path in $\mathcal{I}(v)$ must be a subgraph of $\mathcal{I}'(v)$, so $\mathcal{I}' \geq \mathcal{I}$. 
\end{proof}
By the same argument, one can check that the operation $\vee$ is idempotent, commutative and associative. We make use of associativity in the proof of the following proposition.

\begin{prop}\label{prop:cf.join} $\mathcal{CF}[n]$ is a join semilattice generated by complete connected functions, which are precisely its sets of join-irreducibles.
\end{prop}
\begin{proof} The join operation of Lemma \ref{lem:join} turns $\mathcal{CF}[n]$ into a join semilattice. It is clear that every complete connected function is join-irreducible. Thus it is sufficient to show that they generate $\mathcal{CF}[n]$. 

For $\G \in \mathcal{CF}[n]$, let $L$ be the collection of complete connected function $\H$ such that $\H \leq \G$, and define $\mathcal{I} := \bigvee_{\H \in L}\H$. By definition $\mathcal{I} \leq \G$. To show that $\G \leq \mathcal{I}$, for any pair of nodes $u,v \in [n]$, fix a path $P_{uv}$ in $\G(v)$. We shall construct a complete connected function $\H$ such that $\H(r) \subseteq \G(r)$ for any node $r \in [n]$, and $P_{uv} \subseteq \H(v)$. This implies $\H \in L$, $P_{uv} \subseteq \mathcal{I}(v)$, and hence $\G \leq \mathcal{I}$ by definition of the partial order on $\mathcal{CF}[n]$. We consider two cases.

\textbf{Case 1: $|\lambda(\G)| = 1$}. That is, $\G$ consists of only one part $G$ whose graph is strongly connected. It is sufficient to construct the contraction $\H(v)$, thus we can assume without loss of generality that $G$ contains a self-loop at $v$. Order the nodes so that $v = 1$. Let $\H$ be the output of the following algorithm.
\begin{algorithm}
\begin{algorithmic}[1]
\State Let $G_1, G_2, \ldots, G_n$ be copies of $G$, to be modified in line \ref{ln:prune} below.
\For{$i = 1, \ldots, n$}
	\State Let $H_i$ be an in-directed spanning tree rooted at $i$ of $G_i$. If $i = v = 1$, require that $H_1$ contains $P_{uv}$. 
	\For{$j = i+1, \ldots, n$} 	
		\State Replace the induced subgraph on $G_j$ by the subtree of $H_i$ rooted at $j$. \Comment{Since $ST$ is a spanning tree, the new $G_j$ is still connected.} \label{ln:prune}
	\EndFor
\EndFor
\State Set $\H = (H_1, \ldots, H_n)$.
\end{algorithmic}
\end{algorithm}

By construction the parts of $\H$ are in-directed spanning trees on distinct root nodes, each containing exactly one cycle which is the self-loop at node $v = 1$. Thus Conditions (a) and (c) of Definition \ref{defn:ccr} are satisfied. Condition (b) follows from line \ref{ln:prune}. Thus $\H$ is a complete connected function, and $\H \leq \G$, $P_{uv} \subseteq \H(v)$ by construction.

\textbf{Case 2:} \emph{$|\lambda(\G)| > 1$.} Again, we can assume without loss of generality that there exists a self-loop at $v$. We first apply the above algorithm to the contractions $\tilde{G}_1, \ldots, \tilde{G}_k$, producing $k$ compatible in-directed spanning trees $T_1, \ldots, T_k$ on the strong components of the $G_i$'s. Each edge in these trees fix a source and target node between two strong components. Now we apply the above algorithm to each strong component to split them into compatible spanning trees on singleton sinks. Finally, for each $T_i$, we replace each of its node (which currently is a strong component) by the corresponding spanning tree of that strong component, whose sink is specified by the edges of $T_i$. By construction, the collection of resulting $n$ trees on the $n$ singleton sinks is the desired complete connected function. 
\end{proof}

\subsection{From connected relations to equations and inequalities}
We now define the map which takes an element of ${\mathcal{R}[n]}$ to a cone in $\mathbb{R}^{n \times n}$. We then extend its domain to complete sets of connected relations, resulting in a candidate for the lattice anti-isomorphism of Theorem \ref{thm:buddy}.  

\begin{defn}\label{defn:psi} Let $\psi: G \mapsto \psi(G)$ be the map from ${\mathcal{R}[n]}$ to cones in $\mathbb{R}^{n\times n}$, where $\psi(G)$ is the cone defined by the linear equations and inequalities output by Algorithm ~\ref{alg:psi}.
\end{defn}

\begin{algorithm}
\caption{Definition of $\psi(G)$.}\label{alg:psi}
\begin{algorithmic}[1]
\State Input: $G \in \mathcal{R}[n]$
\State Output: a collection of non-redundant linear equations and inequalities defining a cone $\psi(G)$ in $\R^{n \times n}$
\Procedure{$\psi$}{G}
	\State Pick an in-directed spanning tree $T$ of $G$, let $r$ be its root. Let $T_u$ denote the path from $u$ to the root $r$ in the tree. 
	\State Pick some cycle $C$ in $G$. Let $\lambda(A) = \frac{A(C)}{|C|}$, $\bar{A} = A \odot (-\lambda(A))$.
	\For{each edge $u \to v$ not in $T$}
		\If{edge $u \to v$ is in $G$}
			\State \textbf{print} $\bar{A}_{uv} + \bar{A}(T_v) = \bar{A}(T_u).$ 	($\bar{A}(T_r) = 0$ by convention)
		\Else \,\, \textbf{print} $\bar{A}_{uv} + \bar{A}(T_v) \leq \bar{A}(T_u).$
		\EndIf
	\EndFor
\EndProcedure
\end{algorithmic}
\end{algorithm}

\begin{lem}\label{lem:alg.psi} Algorithm \ref{alg:psi} is well-defined: its output is independent of the choice of $C$ and the spanning tree $T$.
\end{lem}
\begin{proof} 
For any cycle $C' \in G$, the algorithm outputs $\bar{A}(C') = 0$, and hence $A(C)/|C| = A(C')/|C'|$ for all pairs of cycles $C, C' \in G$. Thus the algorithm output is independent of the choice of the cycle $C$. 

To show that the algorithm output is independent of $T$, it is sufficient to prove that for any path $P_{ur}$ from $u$ to $r$, we have $\bar{A}(P_{ur}) = \bar{A}(T_u)$. Indeed, suppose that this is true. Let $T'$ be another in-directed spanning tree of $G$ with root $r'$. Then $r, r'$ are strongly connected, hence $\bar{A}(T_u) = \bar{A}(T'_u) + \bar{A}(T_{r'})$. Thus the linear inequalities and equalities outputted are the same, since the term $\bar{A}(T_{r'})$ cancels on both sides. 

We now prove the above claim. Fix a path $P_{ur}$, and without loss of generality, suppose that $P_{ur} \neq T_u$, and that the first edge of the two paths differ, that is, the edge $u \to v$ on the path $P_{ur}$ is not in $T$. Thus 
$$\bar{A}(T_u) = \bar{A}_{uv} + \bar{A}(T_v) = \bar{A}(P_{ur}(u \to v)) + \bar{A}(T_v).$$
Therefore it is sufficient to prove that $\bar{A}(T_v) = \bar{A}(P_{ur}(v \to r))$ for the induced path from $v$ to $r$ in $P_{ur}$. Repeating this argument, noting that at each step the length of the remaining path is decreasing, we see that $\bar{A}(P_{ur}) = \bar{A}(T_u)$. 
\end{proof}

\begin{lem} Algorithm \ref{alg:psi} outputs the minimal set of equations and inequalities defining the cone $\psi(G)$. This is a closed, non-trivial cone, and its codimension is between $0$ and $n(n-1)$. The cone $\psi(G)$ has codimension 0 if and only if $G$ is a connected function, in which case it is the cone $\Omega_G$ of $\mathcal{P}_n$ as defined in \cite{ST11}. 
\end{lem}
\begin{proof} 
%Since $G \subseteq H \Rightarrow \psi(G) \supseteq \psi(H)$ for $G, H \in \mathcal{R}[n]$, it is sufficient to check triviality for $G$ being the maximal element of $\mathcal{R}[n]$, that is, the complete graph on $n$ nodes with $n$ self-loops. 

The output of Algorithm \ref{alg:psi} is minimal since the $n(n-1)$ equations and inequalities in the output are linearly independent. Indeed, each edge $u \to v$ in $G$ gives rise to the unique equation or inequality in the output that contains the entry $A_{uv}$. 

The n-dimensional lineality space of $\mathcal{N}_n$ is contained in $\psi(G)$ for any $G \in \mathcal{R}[n]$. Thus $\psi(G)$ is a closed, non-trivial cone, with codimension between $0$ and $n(n-1)$. Algorithm ~\ref{alg:psi} outputs no equality if and only if $G$ is a circled tree with tree $T$ and cycle $C$. In this case, the output of Algorithm \ref{alg:psi} is precisely the $n(n-1)$ facet-defining inequalities of the cone $\Omega_G$ found in \cite{ST11}.
\end{proof}

\begin{comment} %--- not sure why we were worried about trivial cones? 
\begin{lem}\label{lem:cr.index} For any $G \in \mathcal{R}[n]$, the cone $\psi(G)$ is non-trivial. That is, $\psi(G) \neq \{\textbf{0}\}$, where $\textbf{0}$ denote the zero matrix. 
\end{lem}
\begin{proof} Since $G \subseteq H \Rightarrow \psi(G) \supseteq \psi(H)$ for $G, H \in \mathcal{R}[n]$, it is sufficient to prove the lemma for $G$ being the maximal element of $\mathcal{R}[n]$, that is, the complete graph on $n$ nodes with $n$ self-loops. It is easy to check that the set of rank-1 matrices of the form
$$A_{ij} = x_i - x_j \mbox{ for some } x \in \mathbb{R}^n $$
satisfy $A \in \psi(G)$. Thus $\psi(G) \neq \{\textbf{0}\}$.  
\end{proof}
\end{comment}

Unfortunately $\psi$ does not completely carry over the lattice structure of $\mathcal{R}[n]$ to its image set of cones. 
\begin{lem}\label{lem:psi.not.im} Suppose $G, H \in \mathcal{R}[n]$. Then $\psi(G \cup H) \subseteq \psi(G) \cap \psi(H)$. Furthermore, $\psi(G \cup H) = \psi(G) \cap \psi(H)$ if and only if either
\begin{enumerate}[(a)]
	\item The sink components are \emph{not} disjoint, or 
	\item There exists a triple of distinct nodes $u, v, w$, $u \in s(G), v \in s(H)$ and paths $P_{wv}, Q_{wu}$ in both $G$ and $H$ such that $Q_{wu} \neq P_{wv}(w \to u)$.
\end{enumerate}	
\end{lem}

\begin{proof} 
The first statement follows from the definition of $\psi$. We prove the second statement by considering three disjoint cases: when condition (a) holds, when (a) does not hold but (b) holds, and when neither (a) nor (b) holds. 

Suppose that (a) holds, that is, $s(G) \cap s(H) \neq \emptyset$. Any inequality in $\psi(G \cup H)$ is also an inequality in both $\psi(G)$ and $\psi(H)$. Thus for $A \in \psi(G) \cap \psi(H)$, we need to show that $A$ satisfies the extra equations in $\psi(G \cup H)$. Let $r \in s(G) \cap s(H)$. For $u \in [n]$, consider the two paths from $u$ to $r$: $G_{ur}$ and $H_{ur}$, coming from $G$ and $H$ respectively. Since $A \in \psi(G)$, $\bar{A}(G_{ur}) \geq \bar{A}(H_{ur})$. Since $A \in \psi(H)$, $\bar{A}(G_{ur}) \leq \bar{A}(H_{ur})$. Thus $\bar{A}(G_{ur}) = \bar{A}(H_{ur})$ for all $u \in [n]$, so $A \in \psi(G \cup H)$. %This proves $\psi(G \cup H) = \psi(G) \cap \psi(H)$ for this case.

Suppose $s(G) \cap s(H) = \emptyset$ and condition (b) holds. By induction on the number of nodes, we reduce to the case $n = 3$ with $s(G) = \{1\}, s(H) = \{2\}$, and and the triple $1,2,3$ witnesses condition (b). That is, $G$ contains the edges $2 \to 1, 3 \to 1, 3 \to 2$, and $H$ contains the edges $1 \to 2, 3 \to 1, 3 \to 2$. Let $A \in \psi(G) \cap \psi(H)$. By definition of $\psi$, $A_{11} = A_{22} = \lambda(A)$. Since $A \in \psi(G)$, $\bar{A}_{31} = \bar{A}_{32} + \bar{A}_{21}$. But $A \in \psi(H)$, hence $\bar{A}_{32} = \bar{A}_{31} + \bar{A}_{21}$. Thus $\bar{A}_{21} + \bar{A}_{12} = 0$, that is, the cycle $C = 1 \to 2 \to 1$ is also a critical cycle of $A$. Thus 
$$\psi(G) \cap \psi(H) = \psi(G \cup C) \cap \psi(H \cup C) = \psi((G\cup C) \cup (H \cup C)) = \psi(G \cup H)$$
where the last equality follows from the first case.

Suppose neither conditions (a) nor (b) hold. By induction on the number of nodes, we can reduce to the case $n = 2$, where $s(G) = \{1\}, s(H) = \{2\}$. By enumerating all possibilities, we can check that regardless of the critical cycle, the matrix
$$A = \left[\begin{array}{cc} 0 & - 1 \\ -1 & 0 \end{array} \right] $$
lies in $\psi(G) \cap \psi(H)$, but not $\psi(G \cup H)$, for example.
\end{proof}

%The failure for $\psi$ to act as a lattice anti-isomorphism on $\mathcal{R}[n]$ is caused by the non-uniqueness of tropical eigenvectors. %Conditions (a) and (b) of Lemma \ref{lem:psi.not.im} are situations where the the tropical eigenvectors of $G$ and $H$ coincide (up to tropical scaling). 
%For most cones this happen when the sink components are disjoint. Condition (a) of Lemma \ref{lem:psi.not.im} thus prevents this case. Condition (b) is the special situation where the two tropical eigenvectors on distinct critical cycles happen to coincide. This explains condition (c) of Definition \ref{defn:ccr}. 

%Lemma \ref{lem:psi.not.im} is the critical observation behind Theorem \ref{thm:buddy}. It states that to encode the cones of linearity of the eigenspace map, we need to extend our combinatorial indexing set to one where the disjoint sink components are kept track of separately, such as $\mathcal{CF}[n]$.
We now show that one can extend $\psi$ to the map $\Psi$ from $\mathcal{CF}[n]$ to cones in $\mathbb{R}^{n \times n}$ satisfying $\Psi(\mathcal{G} \vee \mathcal{H}) = \Psi(\mathcal{G}) \cap \Psi(\mathcal{H})$. This will be our lattice anti-isomorphism between $\mathcal{CF}[n]$ and the fan $\mathcal{F}_n$. 

%================ EXTEND \psi to \CF[n]
\begin{defn}\label{defn:cf.psi} Let $\G = (G_1, \ldots, G_k)$ where each $G_i \in {\mathcal{R}[n]}$. Define a cone in $\R^{n \times n}$ by
$$ \Psi(\G) = \bigcap_{i=1}^k\psi(G_i).$$
\end{defn}

\begin{prop} \label{prop:vee}
For $\G , \H \in \mathcal{CF}[n]$, 
$$ \Psi(\G \vee \H) = \Psi(\G) \cap \Psi(\H).$$
\end{prop}
\begin{proof} The `$\subseteq$' direction follows from Lemma \ref{lem:psi.not.im}. For the other direction, consider the sequence $(B'_1, B'_2, \ldots, B'_k = \emptyset)$ at line \ref{ln:b.prime} of Algorithm \ref{alg:join.on.CF}. Let $(\lambda_1, \ldots, \lambda_k)$ be the corresponding sequence of partition used at the start of line \ref{ln:compute}. We shall prove the `$\supseteq$' direction by induction on $k$. 

If $k = 1$, $B'_1 = \emptyset$. Since the parts of $\G$ and $\H$ are connected, this cannot happen when the union of line \ref{ln:compute} was taken over components with disjoint sinks. Thus $\lambda(\G) = \lambda(\H) = \lambda(\G \vee \H) = \lambda$. By Lemma \ref{lem:psi.not.im}, 
$$
\Psi(\G \vee \H) = \bigcap_{i=1}^{|\lambda|}\psi\left(\bigcup_{u \in \lambda_i}\G(u) \cup \H(u)\right) = \bigcap_{i=1}^{|\lambda|}\bigcap_{u \in \lambda_i}\psi(\G(u)) \cap \psi(\H(u)) = \Psi(\G) \cap \Psi(\H).
$$

If $k = 2$, each cycle $C \in B'_1$ was generated by union of parts in $\G$ and $\H$ that either have non-disjoint sinks, or have the triple-node situation. These are precisely the two cases covered by Lemma \ref{lem:psi.not.im}. Thus for any $A \in \Psi(\G) \cap \Psi(\H)$, $\bar{A}(C) = 0$, and the cone defined by $\Psi(\G) \cap \Psi(\H)$ is unchanged when we adjoin each cycle $C \in B'_1$ to each of the parts of $\G$ and $\H$. Call this new set $\G'$ and $\H'$, that is, define $\G'(u) = \G(u) \bigcup_{C \in B'_1}C, \H'(u) = \H(u) \bigcup_{C \in B'_1}C$. Now suppose that there are nodes $u,v \in C$ which lie in different components of $\lambda_1$. But $u, v \in s(\G'(u)) \cap s(\H'(u))$, thus
$$ \Psi(\G(u)) \cap \Psi(\H(u)) = \Psi(\G'(u)) \cap \Psi(\H'(u)) = \Psi(\G'(u) \cup \H'(u)). $$
Therefore $ \Psi(\G') \cap \Psi(\H') = \Psi(\G) \cap \Psi(\H).$
Since $B'_2 = \emptyset$, it follows that
$$ \Psi(\G \vee \H) = \Psi(\G') \cap \Psi(\H').$$
Combining the last two equations proves the case $k = 2$. Since the algorithm terminates in at most $n^2$ steps, $k \leq n^2$, and induction completes the proof.
\end{proof}

\begin{prop}\label{prop:im} $\Psi$ defines a bijection from the join-irreducibles of $\mathcal{CF}[n]$ to the closed, full-dimensional cones of $\mathcal{F}_n$. 
\end{prop}
\begin{proof} By construction of $\mathcal{F}_n$, its closed, full-dimensional cone are in bijection with a collection of $n$ trees of longest paths on the $n$ roots, each union with a cycle $C$. Note that trees with roots in $C$ coincide. Thus the set of such trees are precisely the complete connected function, which are the join-irreducibles of $\mathcal{CF}[n]$. Finally, $\Psi$ maps a complete connected function to the corresponding cone in $\mathcal{F}_n$ by definition of $\psi$. 
\end{proof}

\emph{Proof of Theorem \ref{thm:buddy} (lattice anti-isomorphism statement).} Propositions \ref{prop:cf.join} and \ref{prop:im} show that $\Psi$ is a bijection from the join-irreducibles of the finitely generated semilattice $\mathcal{CF}[n]$ to the meet-irreducibles of the face lattice of $\mathcal{F}$. It follows from Proposition \ref{prop:vee} that $\Psi$ is a lattice anti-isomorphism. This concludes the proof of Theorem \ref{thm:buddy}. \qed

\emph{Proof of Corollary \ref{cor:buddy}.}\,\,
Let $\G = (G_1, \ldots, G_k)$ be a complete set of connected relations. Let $A \in \Psi(\G)$. By Proposition \ref{prop:kleene}, the sinks $s(G_i), i = 1, \ldots, k$ index distinct tropical vertices of $Pol(A)$. Write $\G = (\D, \mathcal{E})$, where $\D$ is the collection of complete set of connected relations whose sink contain a critical cycle. Thus the corresponding columns of $Pol(A)$ are precisely the extreme tropical eigenvectors of $Eig(A)$ \cite{BCOQ, butkovic}. Therefore, compatible set of connected relations index cones of linearity of the tropical eigenvector map, which are closed cones of $\mathcal{P}_n$.  \qed

\section{Codimension, facet-defining equations and symmetries}\label{sec:codim}
It follows from the proof of Theorem \ref{thm:buddy} that one can obtain the defining equations and inequalities of cones in $\mathcal{F}_n$ from their complete connected relations $\G$ by computing $\Psi(\G)$. One method is to apply Algorithm \ref{alg:psi} to obtain $\psi(G_i)$ for each part $G_i$ of $\G$, then compute their intersections. Here we present a more self-contained algorithm and show how one can quickly compute the codimension of $\Psi(\G)$. The proof of minimal representation similar to that of Algorithm \ref{alg:psi} and hence omitted. 

\begin{algorithm}
\caption{Minimal representation of $\Psi(\G)$.}\label{alg:Psi}
\begin{algorithmic}[1]
\State Input: $\G = (G_1, \ldots, G_k) \in \mathcal{CF}[n]$
\State Output: $\Psi(\G)$
\Procedure{Min-$\Psi$}{$G_1, \ldots, G_k$}
	\For{each $i = 1, \ldots, k$}
		\State \textbf{print} $\psi(S(G_i))$ \Comment{Compute $\psi(S(G_i))$ using Algorithm \ref{alg:psi}}
		\State Compute a spanning tree $T$ of $G_i$. Let $r$ denote its root. \label{ln:ic}
		\For{each edge $u \to v$ not in $T$, $u \in s(G_j), v \in s(G_i)$, $j \neq i$} 
		\If{$u \to v$ in $\widetilde{G}_j$} 
			\State \textbf{print} 
			$\bar{A}_{uv} + \bar{A}(T(v \to r)) = \bar{A}(T(u \to r))$.
		\Else
			\State \textbf{print} 
			$\bar{A}_{uv} + \bar{A}(T(v \to r)) \leq \bar{A}(T(u \to r))$.
		\EndIf
		\EndFor \label{ln:endforic}
	\EndFor
\EndProcedure
\end{algorithmic}
\end{algorithm}

\begin{lem} Algorithm \ref{alg:Psi} outputs a minimal set of equations and inequalities defining the cone $\Psi(G)$, and is independent of the choice of the spanning tree $T$.
\end{lem}

\begin{prop}\label{prop:codim} Let $\G = (G_1, \ldots, G_k) \in \mathcal{CF}[n]$. The codimension of the cone $\Psi(\G)$ is
$$ codim(\Psi(\G)) = \sum_{i=1}^k(e_i - v_i + E_i-V_i) $$
where $v_i, e_i$ are the number of nodes and edges of the graph $S(G_i)$, $V_i$ is the number of nodes in the contraction $\tilde{G}_i$ with an outgoing edge to the node indexed by $s(G_i)$, and $E_i$ is the total out-degree of such nodes. 
In particular, the maximal codimension of $\G$ is $n(n-1)$, and this happens when $G$ is the complete graph on $n$ nodes with $n$ self-loops.  
\end{prop}
\begin{proof} The codimension of $\Psi(\G)$ is the number of equalities returned by Algorithm \ref{alg:Psi}. As in the proof of Proposition \ref{prop:cf.join}, we consider two cases. 

\textbf{Case 1: $k = 1$.} Here $E = V = 1$ and $\Psi(\G) = \psi(G)$. The spanning tree $T$ of $G$ consists of $v-1$ edges, each other edge in $G$ contributes one equality in the output of Algorithm \ref{alg:psi}, except for the special edge $r \to v^\ast$. Thus $codim(\Psi(\G)) = e - (v - 1) - 1 = e - v$. 

\textbf{Case 2: $k > 1$.} By case 1, the number of equalities produced from computing $\psi(S(G_i))$ is $e_i - v_i$. Thus it is sufficient to show that lines \ref{ln:ic} to \ref{ln:endforic} of Algorithm \ref{alg:Psi} yield $\sum_{i=1}^k(E_i - V_i)$ equalities. Suppose $s(G_j)$ has an outgoing edge to $s(G_i)$ in $\tilde{G}_i$. If this is its unique outgoing edge, then it must be in the spanning tree $T$, yielding no equality. Otherwise, since $\tilde{G}_i$ is weakly connected, each outgoing edge from $s(G_j)$ yields a unique alternative path from $s(G_j)$ to $s(G_i)$. Hence the number of equalities each such $s(G_j)$ contributes is exactly its number of outgoing edges minus 1. Taking the sum over all such $j$, we obtain $\sum_{i=1}^k(E_i - V_i)$. 
\end{proof}

\subsection{Symmetries of the fan $\mathcal{F}_n$}
The automorphism group of the complete directed graph on $n$ vertices consists of vertex permutations and reversing edge orientations. Vertex permutations induce symmetries on the fan $\mathcal{F}_n$, mapping one open cone to another. For edge reversals, note that $Pol(A) = Pol(A^T)$ up to vertex permutations\cite{JoswigK10, DS04}, thus edge reversals coincide with some vertex permutations. In terms of $\mathcal{CF}[n]$, edge reversals correspond to tracking all-pairs longest paths by the sources rather than the sinks. If we take $\G \in \mathcal{CF}[n]$, reverse all arrows, and then group the paths by their sink, we obtain another element $\H \in \mathcal{CF}[n]$, and $\G = \H$ up to vertex permutations. This is a non-obvious symmetry of the fan $\mathcal{F}_n$. See Example \ref{ex:n3} below.

\section{Examples}\label{sec:ex}
\begin{ex}[$n = 2$, continued] In this case $\mathcal{N}_2 = \mathcal{P}_2 = \mathcal{F}_2$, and the cones indexed by complete sets of connected functions are shown in Figure \ref{fig:n2}. We shall refer to the cones by their critical cycles. Along the cone $C_{11,22}$ defined by $A_{11} = A_{22} > (A_{12} + A_{21})/2$, $dim(Eig(A)) = 1$, and the two extreme tropical eigenvectors has types shown in the upper-left box of Figure \ref{fig:n2}. Even in this small example we can see Lemma \ref{lem:psi.not.im} in action: matrices in the relative interior of the cone $C_{11,22}$ are precisely the matrices in $\psi(C_{11}) \cap \psi(C_{22})$ but not $\psi(C_{11} \cup C_{22})$, and this is precisely due to the lack of uniqueness of the tropical eigenvector for matrices in this region. 
\end{ex}

\begin{ex}[$n=3$]\label{ex:n3} The $f$-vector of $\mathcal{F}_3$ is $(68, 207, 267, 186, 72, 14, 1)$. As a sanity check, note that $\mathcal{F}_3$ has a lineality space of dimension 3. Identify it with a complete pointed fan in $\R^6$, we see that the Euler characteristic of the $f$-vector of $\mathcal{F}_3$ excluding the point of codimension 6 should be that of a 5-sphere, and it is indeed 0.

Figure \ref{fig:f3.table} shows $N(k,\lambda, p,c)$, the number of cones in $\mathcal{F}_3$ of a given dimension $k$ whose indexing complete set of connected relations $\G$ satisfy $\lambda(\G) = \lambda$, with $p$ number of equalities coming from paths, and $c$ number of equalities coming from cycles. For example, $N(2,(1,1,1),1,1) = 90$ is the number of cones in $\mathcal{F}_3$ which has critical cycles of length at most 1, codimension 2 and thus 2 defining equations, out of which 1 comes from a pair of distinct paths to a critical vertex, and 1 comes from the existence of two critical cycles. For $\lambda(\G) = (3)$, that is, $\G$ has a 3-cycle, we omit $p$ and $c$ since any equality can be regarded as a cycle equality. 

\begin{figure}[h]
	\begin{center}
	 \includegraphics[width =\textwidth]{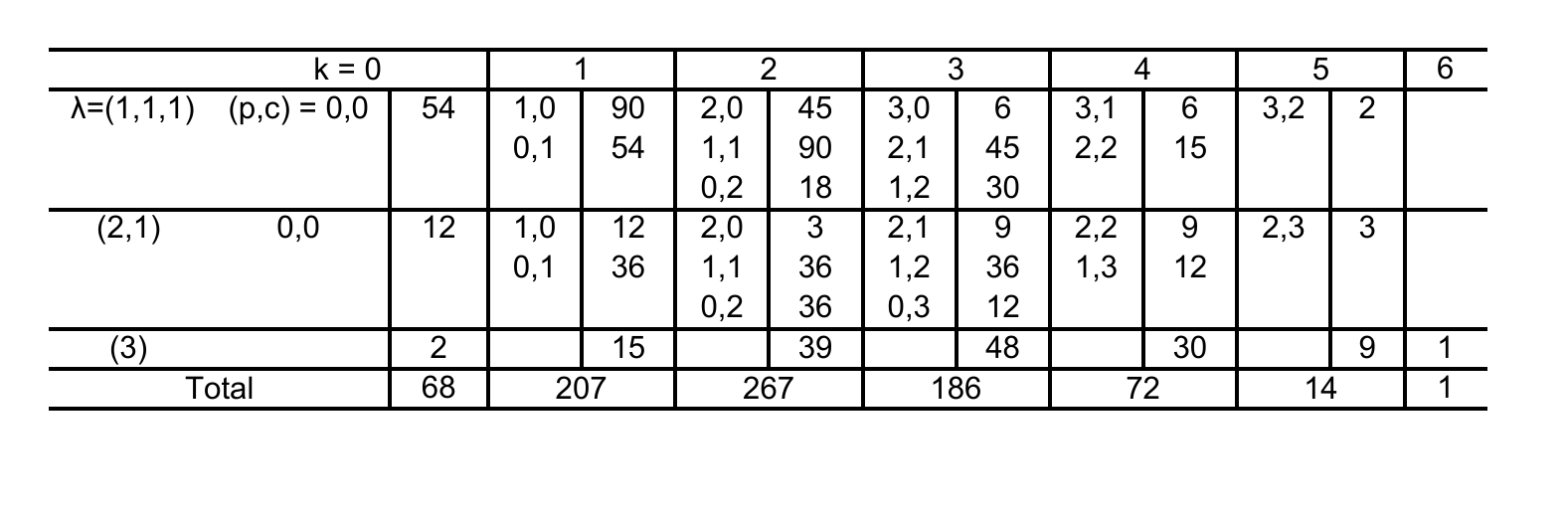}
	 \vskip -1.2cm
\caption{Table of $N(k, \lambda, p,c)$, the number of cones in $\mathcal{F}_3$ of a given dimension $k$ whose indexing complete set of connected relations $G$ satisfy $\lambda(G) = \lambda$, with $p$ the number of equalities coming from paths, and $c$ the number of equalities coming from cycles.}
\label{fig:f3.table}
	\end{center}
\end{figure}
\end{ex}

The 54 cones with partition $\lambda = (1,1,1), p = 0, c = 0$ refine the three cones of $\mathcal{N}_3$ corresponding to the self-loops at $1, 2$ and $3$. Ignoring the self-loops, we have 18 cones, with 5 equivalence classes up to permutations (and edge reversals). These equivalence classes and their sizes are shown in Figure \ref{fig:equiv} below.

Suppose we reverse all edges of $\G_3$. Then we obtain three out-directed trees. Re-organize the paths in these trees by their sinks, we obtain back $\G_3$. Thus edge reversal acts trivially on the orbit of $\G_3$. For $\G_4$, edge reversal coincides with the permutation $(1 \, 2)$.  

\begin{center}
\begin{figure}[h]
	 \includegraphics{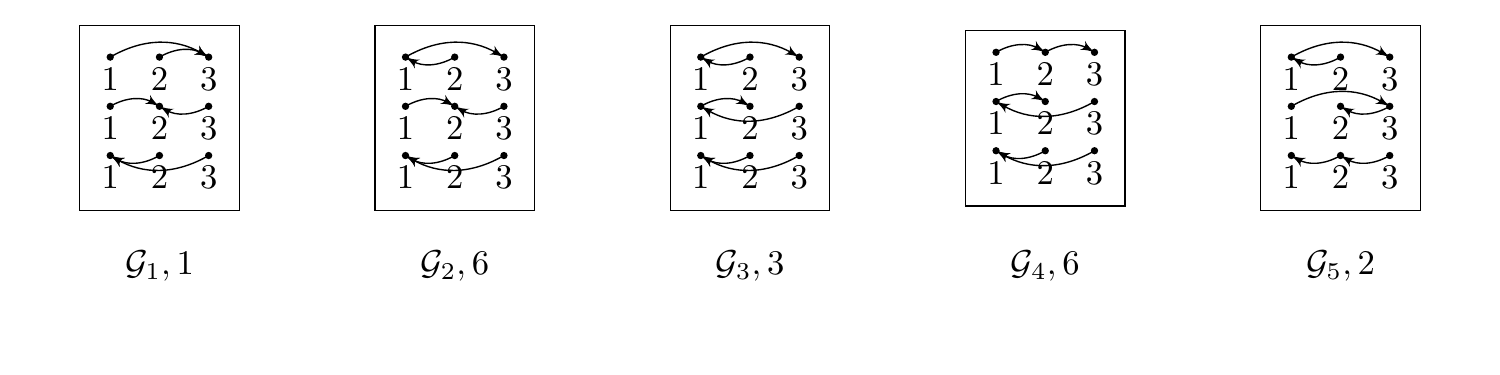}
	 \vskip -1cm
\caption{The five equivalence classes and their sizes, $\lambda = (1,1,1), p = c = 0$}
\label{fig:equiv} 
\end{figure}
\end{center}

\begin{ex}[Face lattice of a full-dimensional cone] We enumerate the face lattice of a full-dimensional cone in $\mathcal{F}_3$ indexed by the complete connected function $\G$ shown in the top solid box of Figure \ref{fig:face.lattice}. The $f$-vector of this cone is $1,6,15,20,15,6,1$. It is a cone over a 5-dimensional simplex with a three-dimensional lineality space.
Algorithm \ref{alg:Psi} gives the following set of defining equations and inequalities for $\Psi(\G)$. For clarity we express them as formulas in $\bar{A}$ instead of $A$
\begin{align*}
\bar{A}_{12} + \bar{A}_{21} &\geq \bar{A}_{33}, &\hspace{1em} \bar{A}_{23} &\geq \bar{A}_{21} + \bar{A}_{13}, \\
\bar{A}_{12} + \bar{A}_{21} &\geq \bar{A}_{11}, &\hspace{1em}
\bar{A}_{22} &\leq \bar{A}_{12} + \bar{A}_{21}, &\hspace{1em}
\bar{A}_{31} &\geq \bar{A}_{32} + \bar{A}_{21}. 
\end{align*}
The face lattice of $\Psi(\G)$ indexed by complete connected relations is displayed as solid graphs in Figure \ref{fig:face.lattice}. Blue arrows indicate subset inclusion. Red edges are those which belong to a cycle in the sink of the corresponding part. The six full-dimensional cones of $\mathcal{F}_3$ adjacent to $\Psi(\G)$ are shown in lighter print next to $\G$. 
\end{ex}

\begin{comment}
Not all open cones in $\mathcal{F}_n$ are simplicial for $n \geq 3$. Below is an example for $n = 3$, which can be generalized to higher dimensions.
\begin{ex}[A non-simplicial open cone of $\mathcal{F}_3$] 
Consider the open cone in $\mathcal{F}_3$ indexed by the complete connected relation $G$ shown in the middle box of Figure \ref{fig:non.simplicial}. This cone has 8 facets, thus it cannot be a simplicial cone.
\begin{figure}[h]
	\begin{center}
	 \includegraphics[width=\textwidth]{graph5.pdf}
	 \vskip -1.3cm
\caption{Facets of a non-simplicial open cone in $\mathcal{F}_3$ labeled by $\G$, the complete connected relation shown in the middle box.  }
\label{fig:non.simplicial}
	\end{center}
\end{figure}
\end{ex}
\end{comment}
\newpage
\begin{landscape}
\thispagestyle{empty}
\begin{figure}[h]
\centerline{\includegraphics[scale = 0.65]{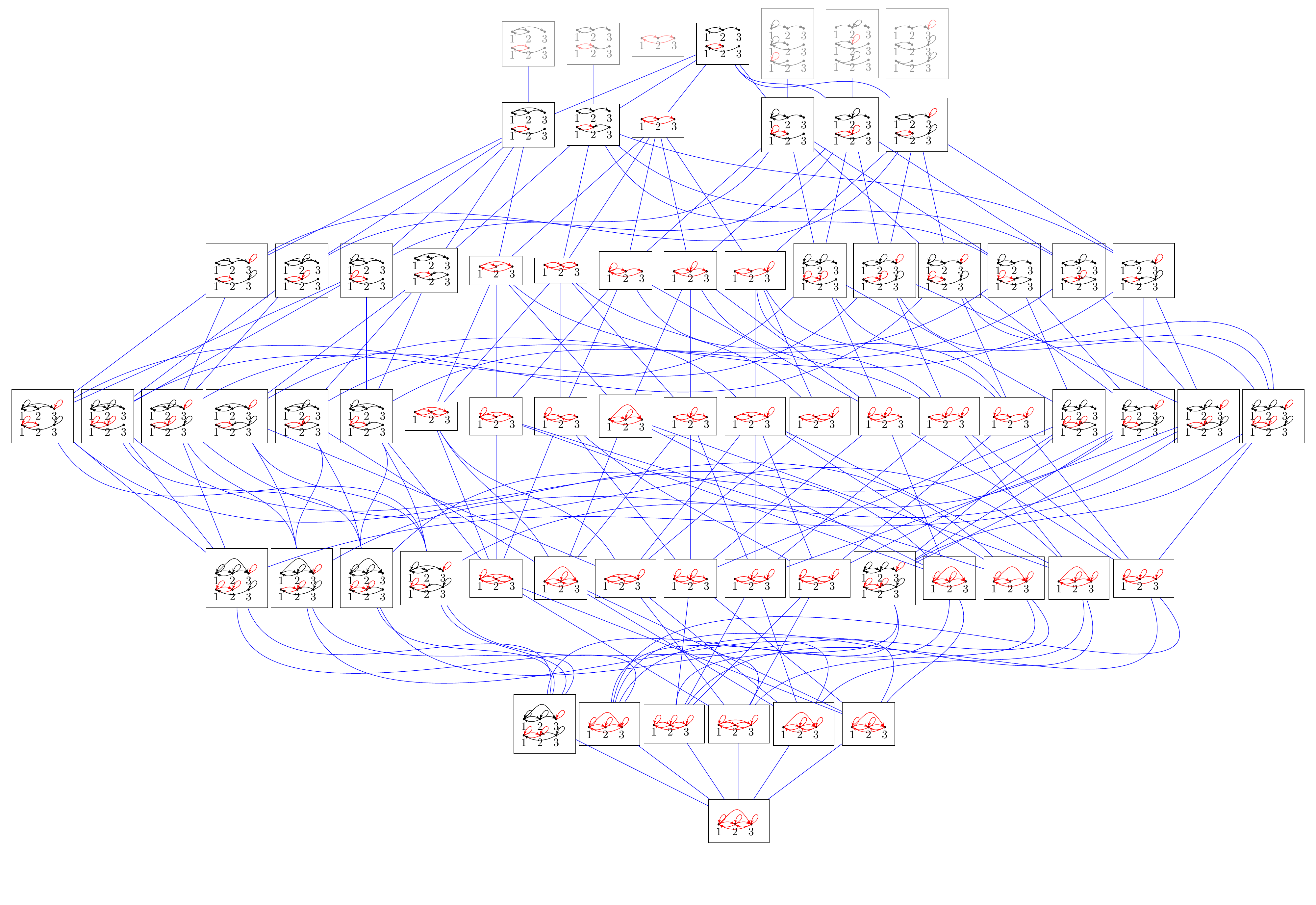}}
	 \vskip -1.5cm
\caption{The face lattice of $\Psi(G)$.}
\label{fig:face.lattice}
\end{figure}
\end{landscape}
\newpage

%-- TODO: discuss this with Bernd / other people in private
%Polytropes appeared in a number of contexts and play an important role in tropical polyhedral geometry, see \cite{JoswigK10} and references therein. As noted in \cite{JoswigK10}, it follows from a result of Develin and Sturmfels \cite[Lemma 10]{DS04} that all polytropes arise as $Pol(A)$ of some square matrix. This result is also known to Sergeev, Schneider and Butkovic \cite{SSB09}, and further studied in \cite{SS08}. General tropical polytopes are also known as semi-modules in the tropical max-plus semi-ring $(\mathbb{R} \cup \{-\infty\}, \oplus, \odot)$, and have been studied from a group theory viewpoint \cite{kambite}. 
\newpage
\bibliographystyle{plain}
\bibliography{references}
\end{document}